\documentclass[11pt,oneside,english]{amsart}
\usepackage[T1]{fontenc}
\usepackage[latin9]{inputenc}
\usepackage{textcomp}
\usepackage{amsthm}
\usepackage{amssymb}
\usepackage{esint}

\makeatletter
\numberwithin{equation}{section}
\numberwithin{figure}{section}
\theoremstyle{plain}
\newtheorem{thm}{\protect\theoremname}[section]
  \theoremstyle{plain}
  \newtheorem*{conjecture*}{\protect\conjecturename}
  \theoremstyle{plain}
  \newtheorem{cor}[thm]{\protect\corollaryname}
  \theoremstyle{plain}
  \newtheorem{lem}[thm]{\protect\lemmaname}
  \theoremstyle{definition}
  \newtheorem{example}[thm]{\protect\examplename}
  \theoremstyle{remark}
  \newtheorem{rem}[thm]{\protect\remarkname}

\def\makebbb#1{
    \expandafter\gdef\csname#1\endcsname{
        \ensuremath{\Bbb{#1}}}
}\makebbb{R}\makebbb{N}\makebbb{Z}\makebbb{C}\makebbb{H}\makebbb{E}\makebbb{H}\makebbb{P}\makebbb{B}\makebbb{Q}\makebbb{E}

\usepackage{babel}

\usepackage{babel}

\makeatother

\usepackage{babel}

\makeatother

\usepackage{babel}
  \providecommand{\conjecturename}{Conjecture}
  \providecommand{\corollaryname}{Corollary}
  \providecommand{\examplename}{Example}
  \providecommand{\lemmaname}{Lemma}
  \providecommand{\remarkname}{Remark}
\providecommand{\theoremname}{Theorem}

\begin{document}

\title{The volume of Kähler-Einstein Fano varieties and convex bodies }

\author{Robert J. Berman, Bo Berndtsson}

\email{robertb@chalmers.se, bob@chalmers.se}

\curraddr{Mathematical Sciences - Chalmers University of Technology and University
of Gothenburg - SE-412 96 Gothenburg, Sweden }
\begin{abstract}
We show that the complex projective space $\P^{n}$ has maximal degree
(volume) among all $n-$dimensional Kähler-Einstein Fano manifolds
admitting a holomorphic $\C^{*}-$action with a finite number of fixed
points. The toric version of this result, translated to the realm
of convex geometry, thus confirms Ehrhart's volume conjecture for
a large class of rational polytopes, including duals of lattice polytopes.
The case of spherical varieties/multiplicity free symplectic manifolds
is also discussed. The proof uses Moser-Trudinger type inequalities
for Stein domains and also leads to criticality results for mean field
type equations in $\C^{n}$ of independent interest. The paper supersedes
our previous preprint \cite{b-b1b} concerning the case of toric Fano
manifolds.

\tableofcontents{}
\end{abstract}
\maketitle

\section{Introduction}

\subsection{Complex geometry}

Let $X$ be an $n-$dimensional complex manifold $X$ which is \emph{Fano,}
i.e. its first Chern class $c_{1}(X)$ is ample (positive) and in
particular $X$ is a projective algebraic variety. For some time it
was expected that the top-intersection number $c_{1}(X)^{n},$ also
called the (anti-canonical)\emph{ degree} of $X,$ is maximal for
the $n-$dimensional complex projective space, i.e. 
\begin{equation}
c_{1}(X)^{n}\leq(n+1)^{n},\label{eq:conj ineq}
\end{equation}
There are now counterexamples to this bound. For example, as shown
by Debarre (see page 139 in \cite{de}), even in the case when $X$
is \emph{toric} (i.e. $X$ admits an effective holomorphic action
of the complex torus $(\C^{*})^{n}$ with an open dense orbit) there
is no universal polynomial upper bound on the $n-$th root of the
degree of $X.$ A more recent conjecture says that the bound above
holds for any Fano manifold whose Picard number is one \cite{kol}.
Given the special role of Kähler-Einstein metrics in complex geometry
- in particular in connection to Chern number inequalities \cite{yau}
- it is also natural to ask if the bound above holds for any Fano
manifold admitting a Kähler-Einstein metric $\omega$? Then $c_{1}(X)^{n}/n!$
is the volume of $X$ in the metric $\omega.$ One step in this direction
was taken by Gauntlett-Martelli-Sparks-Yau \cite{g-m-s-y}, who showed,
using Bishop's volume inequality, that if $X$ is a Fano Kähler-Einstein
manifold then the inequality \ref{eq:conj ineq} holds when the right
hand side is multiplied by $(n+1)/I(X),$ where $I(X)$ is the \emph{Fano
index }of $X,$ i.e. the largest positive integer $I$ such that $c_{1}(X)/I$
is an integral class in the Picard group of $X.$ As is well-known
$I(X)\leq n+1$ with equality precisely for $X=\P^{n}$ and hence
the latter result leaves the question of the maximization property
of $\P^{n}$ open. The main result in this paper shows that the inequality
\ref{eq:conj ineq} indeed holds for Kähler-Einstein Fano manifolds
in the presence of a certain amount of symmetry:
\begin{thm}
\label{thm:max vol vector field intro}Let $X$ be a Fano manifold
which admits a Kähler-Einstein metric and a holomorphic $\C^{*}-$action
with a finite number of fixed points. Then the first Chern class $c_{1}(X)$
satisfies the following upper bound 
\[
c_{1}(X)^{n}\leq(n+1)^{n}
\]
In other words, the complex projective space $\P^{n}$ has maximal
degree among all Fano manifolds $X$ as above.
\end{thm}
The starting point of the proof of the theorem is the fact that, under
the assumptions in the theorem, there is a holomorphic $S^{1}-$action
on $X,$ preserving the Kähler-Einstein metric and with an attractive
fixed point $p.$ The key point of the proof, which builds on our
previous work, is then to study $S^{1}-$invariant Moser-Trudinger
type inequalities in a sufficently large $S^{1}-$invariant Stein
domain $\Omega$ in $X$ containing the fixed point $p$ (compare
section \ref{sub:Critical-mean-field} below). 

The simplest class of varieties in which the assumption in the previous
theorem are satisfied is the class of (generalized)\emph{ flag varieties,}
i.e. rational\emph{ $G-$homogeneous spaces}. As is well-known these
are all Fano manifolds and by homogeneity they also carry Kähler-Einstein
metrics, invariant under the maximal compact subgroup $K$ of $G.$
In this case the bound in the previous theorem was first obtained
by Snow \cite{sn}, using representation theory and quite elaborate
calculcations. 

More generally, the previous theorem applies to any Fano manifold
$X$ on which a reductive connected complex algebraic group $G$ (i.e.
$G$ is the complexification of compact Lie group $K)$ acts algebraically
with finitely many orbits (see Remark \ref{rem:finite}). A particularly
rich class of such $G-$varieties is given by \emph{spherical varieties
}(i.e. a Borel subgroup $B$ of $G$ has an open dense orbit in $X)$
\cite{l-v,bri-1,br-0}. In case a spherical variety is Fano it may
or may not admit a Kähler-Einstein metrics and the inequality in the
previous theorem can hence be viewed as a new obstruction for the
existence of a Kähler-Einstein metric on spherical Fano varities.
According to a formula of Brion \cite{bri-1} the top-intersection
number $c_{1}(L)$ of a polarized spherical variety $(X,L)$ can be
expressed as an explicit integral over a certain polytope $P$ naturally
associated to $X.$ We will recall the symplecto-geometric description
of Brion's formula in section \ref{sec:Symplectic-geometry,-multiplicity}.
Let us also point out that an interesting classical subclass of spherical
varieties is offered by \emph{Schubert varieties }and it is well-known
that any smooth Schubert variety in a Grassmannian is Fano \cite{w-y}.
Another rich subclass is given by $G-$equivariant compactifications
of\emph{ symmetric spaces }and in particular the so called \emph{wonderful
compactifications,} which are often Fano \cite{ru}. In fact, smooth
wonderful compactifications are always \emph{weakly Fano, }i.e.$-K_{X}$
is nef and big \cite{ru} and, in fact, Theorem \ref{thm:max vol vector field intro}
is still valid when $X$ is merely weakly Fano if one uses the notion
of (singular) Kähler-Einstein metrics introduced in \cite{bbegz}
(see Remark \ref{rem:weakly Fano}). 

Of course, in Theorem \ref{thm:max vol vector field intro} it is
enough to assume that $X$ can be \emph{deformed} to a complex manifold
satisfying the assumptions in the theorem. Moreover, in the absense
of a Kähler-Einstein metric we show that the inequality in the Theorem
\ref{thm:max vol vector field intro} still holds when the right hand
side is multiplied by $1/R(X)^{n},$ where $R(X)$ is the greatest
lower bound on the Ricci curvature of $X$ (see Theorem \ref{thm:R(X)}).

\subsubsection*{Toric Fano varieties}

In the special spherical case when $X$ is a \emph{toric} manifold,
i.e. the groups $G$ and $B$ both coincide with the complex torus
$\C^{*n},$ the inequality in Theorem \ref{thm:max vol vector field intro}
was conjectured to hold by Nill-Paffenholz \cite{n-p}. We expect
that the previous theorem can be extended to \emph{singular} spherical
Fano varieties admitting (singular) Kähler-Einstein metrics, but we
will only show this for toric varieties. First recall that, by definition,
$X$ is a Fano variety if $K_{X}$ is an ample $\Q-$line bundle and
$\omega$ is a singular Kähler-Einstein metric on $X$ if its is a
bona fide Kähler-Einstein metric on the regular locus of $X$ such
that $\omega$ extends to a global current in $c_{1}(X)\in H^{2}(X,\Q)$
with continuous local potentials (see \cite{bbegz}). 
\begin{thm}
\label{thm:max vol intro toric sing}Let $X$ be an $n-$dimensional
toric Fano variety which admits a (singular) Kähler-Einstein metric.
Then its first Chern class $c_{1}(X)$ satisfies the following upper
bound which is attained when $X$ is the complex projective space
$\P^{n}:$ 
\[
c_{1}(X)^{n}\leq(n+1)^{n}
\]

\end{thm}
The universal bound in the previous theorem should be contrasted with
the well-known fact that the volume of a general Fano variety $X$
of a fixed dimension $n\geq2$ can be arbitrarily large unless conditions
on the singularities of $X$ are imposed (see \cite{h-m} and references
therein and example 4 in \cite{de2:} for a simple toric example). 

According to the fundamental Yau-Tian-Donaldson conjecture in Kähler
geometry the existence of a Kähler-Einstien metric on a Fano manifold
$X$ is equivalent to $X$ being\emph{ K-stable }(see the recent survey
\cite{p-s}).\emph{ }This notion of stability is of an algebro-geometric
nature. The case of toric Fano manifolds was settled by Wang-Zhou
\cite{w-z}, who more precisely showed that a toric Fano manifold
$X$ admits a Kähler-Einstein metric precisely when $0$ is the barycenter
of the canonical lattice polytope $P_{X}$ associated to $X$ (see
below). In the paper \cite{b-b-2} we extend the result of Wang-Zhou
to the setting of general (possibly singular) Fano varieties:
\begin{thm}
\label{thm:-existence of ke intro}\cite{b-b-2} Let $X$ be an $n-$dimensional
toric Fano variety. Then the following is equivalent:
\begin{itemize}
\item $X$ admits a (singular) Kähler-Einstein metric
\item The barycenter is the unique interior lattice point of the polytope
$P_{X}$ associated to $X.$ 
\end{itemize}
\end{thm}
More generally, the result is shown to hold in the setting of toric
\emph{log Fano varieties }$(X,\Delta)$ familiar from the Minimal
Model Program (MMP), i.e. $X$ is a toric variety and $\Delta$ is
a torus invariant $\Q-$divisor on $X$ with coefficents $<1$ such
that the anti-canonical divisor $-(K_{X}+\Delta)$ of $(X,\Delta)$
defines an ample $\Q-$line bundle on $X.$ In this general setting
Theorem \ref{thm:max vol intro toric sing} holds for the log first
Chern class $c_{1}(-(K_{X}+\Delta))$ of $(X,\Delta)$ if the coefficients
of are positive $\Delta$ and Theorem \ref{thm:-existence of ke intro}
holds for any toric log Fano variety $(X,\Delta).$ 

The barycenter condition for the existence of a Kähler-Einstein metric
on a toric variety is the link to Ehrhart's volume conjecture in convex
geometry, to which we next turn.

\subsection{Convex geometry}

There is a well-known dictionary relating toric polarized varieties
$(X,L)$ and rational polytopes $P$ \cite{de2:,do,c-l-s}. In particular,
the top intersection number $c_{1}(L)^{n}$ coincides with $n!$ times
the volume of the corresponding polytope $P.$ As pointed out in \cite{n-p}
one of the motivations for the bound \ref{eq:conj ineq} on $c_{1}(X)^{n}$
in the\emph{ toric }setting is another more general conjecture of
Ehrhart in the realm of convex geometry, which can be seen as a variant
of Minkowski\textquoteright{}s first theorem for non-symmetric convex
bodies:
\begin{conjecture*}
(Ehrhart). Let $P$ be an n-dimensional convex body which contains
precisely one interior lattice point. If the point coincides with
the barycenter of $P$ then 
\[
\mbox{Vol}(P)\leq\frac{(n+1)^{n}}{n!}
\]

\end{conjecture*}
The case when $n=2$ was settled by Ehrhart \cite{e}, as well as
the special case of simplices in arbitrary dimensions \cite{e2}.
As explained in the survey \cite{g-w} the best upper bound in Ehrhart's
conjecture, to this date, is $\mbox{Vol}(P)\leq(n+1)^{n}/n^{n}.$
As we will next explain Theorem \ref{thm:max vol intro toric sing}
(or rather its more general Log version) confirms Ehrhart's conjecture
for a large class of rational polytopes. First recall that any rational
polytope in $\R^{n},$ containing zero in its interior, may be written
uniquely as 
\begin{equation}
P=\{p\in\R^{n}:\,\,\left\langle l_{F},p\right\rangle \geq-a_{F}\},\label{eq:real polytope intro}
\end{equation}
where the index $F$ ranges over the facets of $P,$ $a_{F}$ is a
positive rational number and the vector $l_{F}$ is a primitive lattice
vector, i.e. it has integer coefficients with no common factors (geometrically,
$l_{F}$ is the inward normal vector of the facet $F$ normalized
with respect to the integer structure). As is well-known \cite{de2:}
toric Fano varieties with $-K_{X}$ an ample $\Q-$line bundle correspond
to rational polytopes $P$ as above with $a_{F}=1.$ More generally,
toric log Fano varieties $(X,\Delta)$ with $\Delta$ an effective
$\Q-$divisor on $X$ correspond to polytopes $P$ with $a_{F}\leq1$
\cite{c-l-s,b-b-2}. Combining the Log version of Theorem \ref{thm:max vol intro toric sing}
above with the existence result for Kähler-Einstein metrics on (possibly
singular) toric varieties hence gives the following
\begin{cor}
\label{cor:ehrhart intro}The bound in the Ehrhart conjecture holds
for all rational polytopes satisfing $a_{F}\leq1$ in the representation
\ref{eq:real polytope intro} and such that $0$ is the barycenter
of $P.$
\end{cor}
In any polytope such that \textbf{$a_{F}\leq1$} the origin is indeed
the unique lattice point (see the foot note on page 105 in \cite{de2:}).
The class of such rational polytopes $P$ is vast and appears naturally
both in algebraic geometry and in combinatorics. For example, the
dual (polar) $P=Q^{*}$ of any lattice polytope $Q$ containing $0$
in its interior is in this class (see \cite{f-k} for the combinatorics
of such polytopes). In particular, this shows that the bound in the
Ehrhart conjecture holds for any \emph{reflexive} convex lattice polytope
(i.e. a lattice polytope $P$ containing $0$ such that its dual $Q$
is also a lattice polytope). Such polytopes correspond to \emph{Gorenstein
}toric Fano varieties and were introduced and studied by Batyrev \cite{ba}
in connection to mirror symmetry of pairs of Calabi-Yau manifolds.
In dimension $n\leq8$ the Ehrhart conjecture up to $n\leq8$ for
reflexive \emph{Delzant} polytopes (i.e. those corresponding to \emph{smooth}
toric Fano varieties) has previously been confirmed by computer assistance
(as announced in \cite{n-p}), using the classification of such polytopes
for $n\leq8$ \cite{oe}).

As explained in section \ref{sec:The-volume-of-singular-toric} the
arguments in the proof of Theorem \ref{thm:max vol intro toric sing}
and its Corollary above can be carried out directly in terms of convex
analysis in $\R^{n}$ without any reference to the corresponding toric
variety. In fact, it is enough to assume that $P$ is a convex body
and one then obtains the following
\begin{thm}
\label{thm:max vol for convex body intro}Let $P$ be a convex body
contained in the positive octant with barycenter $b_{P}=(1,1,...,1).$
Then the volume of $P$ is maximal when $P$ is a regular simplex,
i.e. for $(n+1)$ times the unit-simplex: 
\[
\mbox{Vol}(P)\leq\frac{(n+1)^{n}}{n!}
\]

\end{thm}
Applying the previous theorem to a suitable affine transformation
of a given rational polytope gives us back \ref{cor:ehrhart intro}.

After we had proved the Theorem \ref{thm:max vol for convex body intro}
by the arguments outlined above, Bo'az Klartag showed us a short and
elegant direct proof of this statement using tools from convex geometry.
His argument, that we will reproduce below as Remark \ref{rem:grun},
is based on Grunbaum's inequality \cite{gr}. Nevertheless, we have
decided to keep here our original argument as well since it exemplifies
the relation of this kind of inequalities in convex geometry to Kähler
geometry. Also, our argument may also be useful when dealing with
other singular varities than toric ones, such as spherical varities.

Interestingly, the proof of Grunbaum's inequality is based on a clever
application of the Brunn-Minkowski inequality for convex bodies and
the latter inequality, or more precisely its functional form due to
Prekopa \cite{pr} also plays a key role in our proof (since it is
used in the proof of the Moser-Trudinger type inequalities). In fact,
using the positivity of the direct image bundles in \cite{bern1}
we will obtain a complex geometric generalization of Prekopa's result
in the presence of a suitable action of a compact Lie group on a Stein
manifold (Theorem \ref{thm:prekopa type}). This also leads to a generalization
of the Prekopa theorem in $\R^{n}$ to non-compact real symmetric
spaces of independent interest (Corollary \ref{cor:prekopa symmet}).
Since such spaces typically have negatively sectional curve this latter
result appears to be rather intruiging, when contrasted with the results
in \cite{c-m-s}, which demand non-negative Ricci curvature.

\subsection{\label{sub:Critical-mean-field}Critical mean field type equations
for $S^{1}-$invariant domains in $\C^{n}$}

As explained in section \ref{sub:The-volume-of} the starting point
of the proof of Theorem \ref{thm:max vol vector field intro} is that
an $S^{1}-$invariant Kähler-Einstein metric on $X$ induces a solution
$\phi$ to a mean field type equation on a Stein domain $\Omega$
of $X$ admitting an $S^{1}-$action with an attractive fixed point.
The proof is thus reduced to establishing a criticality result for
solutions of such equations. For concreteness, here we will only state
the result in the case when the domain $\Omega$ is contained in $\C^{n}$
with its standard action by $S^{1}$ (or more generally any linear
action of $S^{1}$ with positive weights $m_{i},$ i.e. defined by
$(e^{i\theta},(z_{1},...,z_{n}))\mapsto(e^{im_{1}\theta}z_{1},...,e^{im_{n}\theta}z_{n})).$
Denote by $dV$ the Euclidean volume element and write $dd^{c}=i\partial\bar{\partial}/2\pi,$
so that $(dd^{c}\phi)^{n}$ is the Monge-Ampère measure of $\phi,$
whose density is equal to $(\frac{i}{2\pi})^{n}$ times the determinant
of the complex Hessian $(\frac{\partial^{2}\phi}{\partial z_{i}\partial\bar{z}_{j}}).$ 
\begin{thm}
\label{thm:mean field intro}Let $\Omega$ be a connected smoothly
bounded $S^{1}-$invariant pseudoconvex domain in $\C^{n},$ containing
$0.$ For any given positive real number $\gamma$ a necessary condition
for the existence of an $S^{1}-$symmetric plurisubharmonic solution
$\phi\in\mathcal{C}^{\infty}(\bar{\Omega})$ to the equation \textup{
\begin{equation}
(dd^{c}\phi)^{n}=\frac{e^{-\gamma\phi}dV}{\int_{\Omega}e^{-\gamma\phi}dV}\,\,\mbox{in\,\,\ensuremath{\Omega}}\,\,\,\,\phi=0\,\,\mbox{on\,\ensuremath{\partial\Omega}}\label{eq:mean field type eq intro}
\end{equation}
} is that $\gamma\leq(n+1).$
\end{thm}
More generally, the proof of the previous theorem shows that it is
enough to assume that $\phi$ is a continuous solution in the sense
of pluripotential theory. As shown in \cite{b-b} $\gamma<(n+1)$
is a \emph{sufficient} condition for existence of such weak solutions
for\emph{ any} pseudoconvex domain (by \cite{gky} any such continuous
solution is in fact smooth in the interior of $\Omega).$ Hence $\gamma=n+1$
appears to be a\emph{ critical} parameter for the equations \ref{eq:mean field type eq intro}
on $\Omega$ in the presence of an $S^{1}-$action as above. Moreover,
the condition in the previous theorem that the fixed point $0$ be
contained in $\Omega$ is crucial. For example, if $\Omega$ is an
annulus in $\C$ then it is well-known that there exist $S^{1}-$invariant
solutions for $\gamma$ arbitrarily large (see section 5 in \cite{clmp}).

To prove the theorem we first show that any solution $\phi$ as in
the previous theorem is an extremal for a Moser-Trudinger type inequality
on $\Omega,$ which becomes stronger as the parameter $\gamma$ increases.
We then go on to show that when $\gamma>(n+1),$ a suitable regularization
of the pluricomplex Green function $g$ of $\Omega$ with a logarithmic
pole at $0$ violates the corresponding Moser-Trudinger type inequality.

This approach should be compared with a result of Ding-Tian \cite{d-t}
saying that for \emph{any} Fano manifold $X$ admitting a Kähler-Einstein
metric there is a corresponding Moser-Trudinger type inequality for
positively curved metrics on the anti-canonical line bundle $-K_{X}$
which has the metric on $-K_{X}$ induced by the Kähler-Einsten metric
as an extremal. However, it is well-known that there are global obstructions
to the existence of metrics on $-K_{X}$ with prescribed logarithmic
poles and this is the reason that we need to replace $X$ with a Stein
subdomain $\Omega.$ The prize we have to pay is that an appropriate
symmetry assumption is then needed to deduce the corresponding Moser-Trudinger
type inequalities on $\Omega$ (compare the discussion on symmetry
breaking in section \ref{sub:Symmetry-breaking}).

We conjecture that the inequality $\gamma\leq n+1$ in the previous
theorem is in fact a \emph{strict} inequality and that if $\phi_{j}$
is a sequence of $S^{1}-$symmetric solutions $\phi_{j}$ associated
to a sequence of parameters $\gamma_{j}$ converging to the critical
value $n+1,$ then $\phi_{j}$ converges weakly to the pluricomplex
Green function $g_{\Omega}$ of $\Omega$ with a logarithmic pole
at $0.$ This is easy to verify in the case when $\Omega$ is the
unit-ball and $\phi_{j}$ is the radial solution (see \cite{b-b}).
The motivation for this conjecture comes from the concentration-compactness
principles extensively studied in the one-dimensional situation (see
\cite{clmp} and references therein). As explained in \cite{clmp}
the equations above then appear as mean field equations for statistical
mechanical models with $\gamma$ playing the role of the inverse temperature
and the critical value corresponding to a phase transition.

\subsection{Organization}

After having set up the complex geometric and group-theoretic frame
work in the beginning of Section \ref{sec:The-volume-of-k-e} we establish
Prekopa type convexity inequalites and give the proof of Theorem \ref{thm:max vol vector field intro},
by reducing it to the proof of Theorem \ref{thm:mean field intro}.
Then in Section \ref{sec:The-volume-of-singular-toric} the singular
setting on a toric variety is considered and the proof of Theorem
\ref{thm:max vol intro toric sing} is explained, using real convex
analysis. Finally, in Section \ref{sec:Symplectic-geometry,-multiplicity}
Theorem \ref{thm:max vol vector field intro} is applied to spherical
varieties and rephrased in terms of Lie algebras and symplectic geometry.

\subsection{Acknowledgments}

We are grateful to Benjamin Nill for helpful comments on the toric
setting, Michel Brion for his help with spherical varieties, Gabor
Székelyhidi for encouraging us to consider the relation to the invariant
$R(X)$ and Bo'az Klartag for allowing us to include here his beautiful
reduction to Grunbaum's inequality.

\section{\label{sec:The-volume-of-k-e}The volume of Kähler-Einstein Fano
manifolds}

\subsection{Preliminaries}

\subsubsection{Kähler-Einstein metrics and Monge-Ampère equations}

Let $L\rightarrow X$ be a holomorphic line bundle over an $n-$dimensional
compact complex manifold. We will denote by $H^{0}(X,L)$ the space
of all global holomorphic sections with values in $L.$ A Hermitian
metric $\left\Vert \cdot\right\Vert $ on $L$ may be represented
by a collection of local functions $\phi(:=\{\phi_{U}\})$ defined
as follows: given a local trivializing section $s$ of $L$ on an
open subset $U\subset X$ we define the local weights $\phi_{U}:=-\log\left\Vert s\right\Vert ^{2}$
of the metric. Of course, $\phi_{U}$ depends on $s,$ but the (normalized)\emph{
curvature form }of the metric\emph{ }
\[
dd^{c}\phi_{U}:=\frac{i}{2\pi}\partial\bar{\partial}\phi
\]
 is a globally well-defined two form on $X$ representing the first
Chern class $c_{1}(L).$ The normalizations have been chosen so that
$c_{1}(L)$ is an \emph{integral }class. Note that the metric on $L$
has \emph{semi-positive curvature} form precisely when the local weights
$\phi_{U}$ are \emph{plurisubharmonic} \emph{(psh,} for short). We
recall that according to the Kodaira embedding theorem the line bundle
$L$ is\emph{ ample, }i.e. the Kodaira map $X\rightarrow\P(H^{0}(X,L))^{*}$
is an embedding for $k$ sufficiently large, precisely when $L$ admits
a metric with positive curvature. For any such $k$ the open manifold
$S:=\{s=0\},$ for $s$ a given non-trivial element in $H^{0}(X,kL),$
is a\emph{ Stein manifold}, i.e. $S$ admits a smooth and strictly
plurisubharmonic exhaustion-function $\phi_{S}.$ Indeed, $\phi_{S}$
can be taken as $-\log\left\Vert s\right\Vert ^{2}$ for any positively
curved metric on $L.$

A Kähler metric $\omega$ on $X$ is said to be \emph{Kähler-Einstein}
if it has constant Ricci curvature, i.e. 
\[
\mbox{Ric \ensuremath{\omega=\Lambda\omega}}
\]
 for some constant $\Lambda.$ We will be interested in the case when
$\Lambda$ is positive and after a scaling we may as well assume that
$\Lambda=1.$ Then $X$ is necessarily a \emph{Fano manifold,} i.e.
the dual $-K_{X}$ of the canonical line bundle $K_{X}:=\Lambda^{n}(T^{*}X)$
is ample. Equivalently, $\omega$ is a Kähler-Einstein metric on the
Fano manifold $X$ iff $\omega$ is the curvature of a positively
curved metric $\left\Vert \cdot\right\Vert $ on $-K_{X}$ such that,
if $s$ is a local trivialization of $-K_{X}$ over $U,$ then there
is a positive constant $C$ such that the following Monge-Ampère equation
holds on $U:$ 
\[
(dd^{c}\phi)^{n}=Ce^{-\phi}dV,\,\,\,\, dV:=i^{n^{2}}\theta\wedge\bar{\theta}
\]
 where $\phi$ is the corresponding weight of the metric and $\theta$
is the holomorphic $n-$form which is dual to $s.$ More generally,
given a positive integer $k$ the local section $s$ can be replaced
by a local non-vanishing holomorphic section $s_{k}$ of $-kK_{X}\rightarrow U$
and one then sets $\phi:=-\frac{1}{k}\log\left\Vert s_{k}\right\Vert ^{2},$
replacing the dual $\theta$ above with its $k$th root. We remark
that this latter flexibility allows one to define Kähler-Einstien
metrics on a \emph{singular} (normal) variety $X$ (see \cite{bbegz}).
Indeed, in general $K_{X}$ is then merely defined as a Weil divisor,
but assuming that $X$ is a Fano variety, i.e. $-kK_{X}$ is an ample
line bundle for some positive integer $k$ (in other words $K_{X}$
is an ample $\Q-$Cartier divisor) the previous definition makes sense
on the regular locus of $X$ and one then adds the global condition
that the corresponding metric be\emph{ continuous }on all of $X.$
Anyway, in this paper we will mainly stick to the case when $X$ is\emph{
smooth.}

\subsubsection{Group actions}

Let $G$ be a complex Lie group acting by holomorphisms on a compact
complex manifold $X.$ In other words, $X$ is a compact complex $G-$manifold.
If $G$ acts linearly on a vector space $V,$ we write $V^{G}$ for
the subspace of $G-$invariant vectors and $V^{(G)}$ for the subspace
of $G-$eigenvectors, i.e. $v\in V^{(G)}$ iff $gv=\chi(g)v,$ where
$\chi$ is a \emph{character}, i.e. a homomorphism from $G$ to $\C^{*}.$
In particular, if $L$ is a $G-$equivariant line bundle over $X$
then $G$ acts linearly on the vector space $H^{0}(X,L)$ by setting
$(g\cdot s)(p):=g(s(g^{-1}p))$ for any $s\in H^{0}(X,L).$

In the proof of Theorem \ref{thm:max vol vector field intro} we will
have great use for the following 
\begin{lem}
\label{lem:exist of section}Let $X$ be a smooth projective variety
admitting a holomorphic action by the circle $S^{1}$ with isolated
fixed points. If the first Betti number of $X$ vanishes and $L$
is a given $S^{1}-$equivariant ample line bundle over $X,$ then
there exists a fixed point $p\in X$ and an $S^{1}-$eigenvector $s\in H^{0}(X,kL)$
for some $k>0$ such that $s(p)\neq0$ and 
\begin{equation}
H^{0}(S)^{T}=\C,\label{eq:lemma exist of section}
\end{equation}
 where $S$ is the Stein manifold $S:=X-\{s=0\}$ containing $p.$\end{lem}
\begin{proof}
Since $L$ is ample it admits a metric with positive curvature form
$\omega,$ defining a symplectic form on $X.$ Moreover, averaging
over the compact group $S^{1}$ we may assume that $\omega$ is $S^{1}-$invariant,
i.e. the action is symplectic. Since the first Betti number of $X$
vanishes the action admits a Hamiltonian function $f,$ i.e. $df=\omega(V,\cdot)$
where $V$ is the vector field on $X$ generating the $S^{1}-$action.
It then follows from general principles that the action lifts to $L.$
Anyway, in our setting $L$ will be equal to $-K_{X}$ which admits
a canonical lift of the $S^{1}-$action on $X.$ We let $p$ be a
point where the minimum of $f$ is attained. Then $V$ vanishes at
$f$ and hence $p$ is a fixed point. Next, we pick a positive number
$k$ such that $kL$ is globally generated and decompose $H^{0}(X,kL)=\oplus V_{m}$
in the one-dimensional eigenspaces for the $S^{1}-$action. By the
assumption of global generation there is a section $s\in H^{0}(X,kL)$
such that $s(p)\neq0$ and by the previous decomposition we can thus
take $s\in V_{m}$ for some $m.$ Let now $S:=X-\{s=0\},$ which,
as explained above, is a Stein manifold. To prove \ref{eq:lemma exist of section}
we note that the action of $S^{1}$ on the tangent space at the fixed
point $p$ has\emph{ positive }weights in the following sense (i.e.
it is an \emph{attractive fixed point}) : by a general result for
compact Lie group (see Satz 4.4 in \cite{kau}) we may linearize the
action in an invariant neighborhood of $U$ of $p$ so that $e^{i\theta}\cdot(z_{1},...,z_{n})=(e^{im_{1}\theta}z_{1},...,e^{im_{n}\theta}z_{n}).$
The positivity referred to above then amounts to having $m_{i}>0$
for all $i.$ Indeed, after performing a linear change of coordinates
we may assume that $f(z)=\sum m_{i}|z_{i}|^{2}+o(|z|^{2})$ and since
$f$ has a minimum at $p$ (corresponding to $z=0)$ and the fixed
point $p$ is isolated it must be that $m_{i}>0.$ Taylor expanding
a given holomorphic function $g$ on $U$ wrt the variables $z_{i}$
then reveals that $g$ is $S^{1}-$invariant only if it is constant
on $U.$ But since $S$ is connected this concludes the proof of the
lemma.\end{proof}
\begin{example}
Set $X=\P^{1}$ and $L=-K_{X}$ with its usual $S^{1}-$action obtained
by identifying $X$ with the two-sphere and considering rotations
around a fixed axes. Then $X$ has two fixed points (the north and
the south pole) and fixing the standard affine chart $U_{0}\approxeq\C$
containing the south pole, where the action is given by $(e^{i\theta},z)\mapsto e^{i\theta}z,$
we can write any section $s\in H^{0}(X,-K_{X})$ over $U_{0}$ as
$s=f(z)\frac{\partial}{\partial z}$ so that $(e^{i\theta}\cdot s)(z)=e^{i\theta}f(e^{-i\theta}z)dz.$
Hence, we can take the section $s$ in the previous lemma as the one
determined by $f(z)=1$ (which has weight $m=1)$ so that $S=U_{0}.$ 
\end{example}
Even though if it will not - strictly speaking - be needed for the
proof of Theorem \ref{thm:max vol vector field intro}, we make a
brief digression to explain how, using the Bialynicki- Birula decomposition
(see Theorem 4.4 in \cite{bi-b}), one may, essentially, take the
Stein domain $S$ to be equivariantly isomorphic to $\C^{n}$ with
a linear $\C^{*}-$action. First recall that the Bialynicki- Birula
decomposition says that any non-singular $n-$dimensional projective
variety $X$ with a $\C^{*}-$action having only isolated fixed points
may be written as the disjoint union $X=\coprod_{p}X_{p}$ where $p$
ranges over the fixed point $p$ in $X$ and where $X_{p}$ is $\C^{*}-$
invariant set equivariantly isomorphic to the affine space $T^{+}X_{|p},$
i.e. the the direct sum of the positive weight spaces in $TX_{|p}$
with the induced linear $\C^{*}-$action. Concretely, $X_{p}$ is
the \emph{attracting set} for $p$ under the $\C^{*}-$action , i.e.
$x\in X_{p}$ iff $\lim_{\lambda\rightarrow0}\lambda\cdot x=p.$ In
particular, if the all the weights at $p$ are all positive (as for
$p$ in the previous lemma) then $X_{p}$ is a Zariski open subvariety
of $X$ which is equivariantly isomorphic to $\C^{n}$ with a linear
$\C^{*}-$action. We note that 
\[
X_{p}\subset S,
\]
where $S=\{s\neq0\}$ is the Stein manifold appearing in the previous
lemma. Moreover, given a positively curved $S^{1}-$invariant metric
on $L$ the function $\phi:=-\log\left\Vert s\right\Vert ^{2}$ is
a psh exhaustion function of $X_{p}.$ To see that $\phi$ is indeed
proper we note that since $\phi$ is strictly convex along the $\C^{*}-$orbits
and (since $s(p)\neq0)$ is bounded from below close to $p$ it follows
that $\phi\rightarrow\infty$ as $\left|\lambda\right|\rightarrow\infty$
along a given $\C^{*}-$orbit in $X_{p},$ which implies properness
by a basic compactness argument. As for the $S^{1}-$invariance it
follows from the fact that $s$ is an eigenvector (as explained in
the beginning of section \ref{sub:The-volume-of} below).

\subsection{A Prekopa type convexity result on Stein manifolds under group actions}

One of the key ingredients in the proof of Theorem \ref{thm:max vol vector field intro}
is a convexity result of Prekopa type. Let us first recall the Prekopa
inequality in its original form \cite{pr}: If $\phi_{t}(x)$ is a
convex function on $I\times\R^{n}$, where $I$ is an open interval
in $\R$ with coordinate $t,$ then the function 
\[
t\mapsto-\log\int_{\R^{n}}e^{-\phi_{t}}dx
\]
 is convex. Here we will obtain a complex geometric generalization
of this result. We let $S$ be an $n-$dimensional Stein manifold
with trivial canonical line bundle $K_{S},$ i.e. $S$ admits a non-vanishing
holomorphic $n-$form $\theta$ (also called a \emph{holomorphic volume
form}).
\begin{thm}
\label{thm:prekopa type}Let $K$ be a compact group acting on a bounded
Stein domain $\Omega$ with a holomorphic volume form $\theta$ and
such that all $K-$invariant holomorphic functions are constant, i.e.
\begin{equation}
H^{0}(\Omega)^{K}=\C.\label{eq:prekopa type}
\end{equation}
 If $\phi_{t}(z)$ is a $K-$invariant bounded psh function on $D\times\Omega,$
where $D$ is the unit-disc in $\C$ then the function 
\begin{equation}
t\mapsto-\log i^{n^{2}}\int_{\Omega}e^{-\phi_{t}}\theta\wedge\bar{\theta},\label{eq:statement theorem prek}
\end{equation}
 is subharmonic in $t.$ More generally, if $\Omega$ is unbounded
and $e^{-\phi_{t}}$ is integrable on $\Omega$ for $t$ fixed, the
same conclusion holds if the space $H^{0}(\Omega)$ is replaced by
$H^{0}(\Omega)\cap L^{2}(e^{-\phi_{t}}\theta\wedge\bar{\theta}).$\end{thm}
\begin{proof}
Consider the (infinite dimensional) Hermitian holomorphic vector bundle
$E\rightarrow D$ whose fiber $E_{t}$ is the Hilbert space of all
holomorphic functions on $\Omega$ of finite $L^{2}-$norm 
\[
\left\Vert f\right\Vert _{\phi_{t}}^{2}:=i^{n^{2}}\int_{\Omega}|f|^{2}e^{-\phi_{t}}\theta\wedge\bar{\theta},
\]
 As shown in \cite{bern1} this bundle has \emph{positive curvature}
in the following sense: for any given holomorphic section $\Lambda$
of the dual bundle $E^{*}$ the function 
\[
t\mapsto-\log(\left\Vert \Lambda\right\Vert _{\phi_{t}}^{2}):=-\log(\sup_{f\in E_{t}}\left|\left\langle \Lambda,f\right\rangle \right|^{2}/\left\Vert f\right\Vert _{\phi_{t}}^{2})
\]
is subharmonic. Strictly speaking the proof in \cite{bern1} concerned
the case when $\Omega$ is a pseudoconvex domain in $\C^{n},$ but
the proof can be repeated word for word in the Stein case, using that
Hörmander's $L^{2}-$estimates for $\bar{\partial}$ are still valid.
We now let $\sigma$ be the invariant probability measure on $K$
(i.e. the Haar measure) and set 
\[
\Lambda(f):=\int_{K}(k^{*}f)\sigma
\]
Since the rhs above is a $K-$invariant holomorphic function on $\Omega$
it is, by assumption, constant (for $t$ fixed) and hence $\Lambda$
indeed defines a holomorphic section of $E^{*}.$ Thus it will be
enough to show that, under the condition \ref{eq:prekopa type}, $\left\Vert \Lambda\right\Vert _{\phi_{t}}^{2}=1/\int_{\Omega}e^{-\phi_{t}}\theta\wedge\bar{\theta},$
i.e. the sup above is attained for constant functions. But this follows
immediately from estimating 
\[
\int_{\Omega}e^{-\phi_{t}}\theta\wedge\bar{\theta}\left\Vert \Lambda\right\Vert _{\phi_{t}}^{2}\leq\int_{K}d\sigma\int_{\Omega}e^{-\phi_{t}}|k^{*}f|^{2}\theta\wedge\bar{\theta}=\int_{\Omega}e^{-\phi_{t}}|f|^{2}\theta\wedge\bar{\theta},
\]
 using that $\phi_{t}$ is $K-$invariant in the last equality.
\end{proof}
For example, if $\Omega=S=\C^{n},$ $\theta=dz$ and $\phi_{t}$ grows
as $(n+1)\log(|z|^{2})+O(1)$ at infinity one can take the compact
group $K$ in the theorem to be trivial, i.e. no symmetry assumption
is needed. The point is that, viewing $\C^{n}$ as a Zariski open
set in $X:=\P^{n}$ the space $H^{0}(\C^{n})\cap L^{2}(e^{-\phi_{t}})$
can be identified with $H^{n,0}(X,-K_{X})\approxeq H^{0}(\P^{n},\C)$
which is one-dimensional. However, for a general psh function $\phi_{t}$
in $\C^{n}$ one can construct counter-examples to the subharmonicity
in formula \ref{eq:prekopa type}, by adapting an example of Kiselman
\cite{ki} to the present setting. 
\begin{rem}
Replacing $\phi$ by $m\phi$ and letting $m\rightarrow\infty$ in
the previous theorem shows that the function 
\[
t\mapsto\inf_{\Omega}\phi_{t}
\]
 is also subharmonic in $t.$ This property is a well-known instance
of the Kiselman minimum principle \cite{ki} extended to the setting
of Lie groups by Loeb \cite{lo}. It should however be pointed out
that the setting in \cite{lo} is more general as it applies to certain
non-compact groups $K.$ 
\end{rem}

\subsubsection*{An application to symmetric spaces}

Before continuing we make a brief digression, showing how the previous
theorem implies a Prekopa type convexity inequality for real \emph{symmetric
spaces}. More precisely, we consider a non-compact symmetric space
with compact dual $K,$ i.e. the symmetric space may be written as
$G/K$ where $G$ is the complexification of the compact group $K$
(so that $G$ is a connected reductive complex Lie group and for simplicity
we will assume that $G$ is semi-simple \cite{he}). As is well-known
$G$ is a Stein manifold. More precisely, by results of Chevalley
$G$ is an affine algebraic variety (see the appendix in \cite{ma1})
with a $G$-bi-invariant pseudo-Riemannian metric (induced by the
Killing form on the Lie algebra of $G$). The corresponding $K-$principal
fiber bundle 
\[
\pi:\, G\rightarrow G/K
\]
 induces a $G-$bi-invariant symmetric Riemannian metric on $G/K.$
There is a $G-$bi-invariant volume form $\mu_{G}$ on $G$ and it
can be written as $\mu_{G}=\theta\wedge\bar{\theta},$ where $\theta$
is a $G-$bi-invariant holomorphic top form on $G.$ The push-forward
of $\mu_{G}$ under the projection $\pi$ clearly coincides with the
$G-$bi-invariant volume form $\mu_{G/K}$ on the symmetric space
$G/K.$ 
\begin{cor}
\label{cor:prekopa symmet}Let $\phi_{t}$ be a geodesically convex
function on the Riemannian product $\R\times G/K,$ where $\R$ is
the real line with its Euclidean Riemannian metric. Then the function
\[
t\mapsto-\log\int_{G/K}e^{-\phi_{t}}\mu_{G/K}
\]
is convex on $\R.$\end{cor}
\begin{proof}
It is well-known that the convex functions on the symmetric space
$G/K$ may be written as $\pi_{*}\psi$ where $\psi$ is a $K-$invariant
psh function on $G$ (see for example Lemma 2, page 34 in \cite{do}).
In particular, we may identify $\phi_{t}$ with a $S^{1}\times K-$invariant
psh function $\psi_{t}$ on $\C^{*}\times G.$ Hence, the corollary
will follow from the previous theorem once we have checked that any
$K-$invariant holomorphic function on $G$ is constant. But this
follows immediately from the basic fact that a holomorphic function
on an $n-$dimensional connected complex manifold is uniquely determined
by its restriction to a totally real submanifold (here $K)$ of real
dimension $n.$\end{proof}
\begin{example}
If $G=SL(2,\C)$ and $K=SU(2),$ then the symmetric space $G/K$ may
be identified with the three-dimensional real hyperbolic space.
\end{example}
It is interesting to compare the previous corollary with the results
in \cite{c-m-s}, where a so called \emph{Prekopa-Leindler inequality}
is obtained valid for any (possibly non-symmetric) Riemannian manifold
$(M,g)$ with non-negative Ricci curvature. This inequality is in
fact stronger than the Prekopa inequality on $(M,g),$ i.e. the analog
of Corollary \ref{cor:prekopa symmet} for $(M,g).$ For example,
in the case when $(M,g)$ is Euclidean space (which may be obtain
as above by taking $G$ to be abelian) the Prekopa-Leindler inequality
implies the Brunn-Minkowski inequality for general sets, while the
Prekopa inequality a priori only implies the Brunn-Minkowski inequality
for\emph{ convex} sets. Conversely, as shown in \cite{c-m-s} the
validity of the Prekopa-Leindler inequality on $(M,g)$ actually implies
that the Ricci curvature of $(M,g)$ is non-negative and hence it
cannot hold on a general symmetric space as above, since it is well-known
that an irreducible symmetric space $G/K$ has\emph{ negative }sectional
curvatures when $G$ is non-abelian.

\subsection{\label{sub:The-volume-of}Proofs of Theorems \ref{thm:max vol vector field intro}
, \ref{thm:mean field intro}}

We will start by reducing the proof of Theorem \ref{thm:max vol vector field intro}
to proving Theorem \ref{thm:mean field intro}. In the following we
will use the notation $T=S^{1}$ for the one-dimensional real torus
inbedded in $\C^{*}$ and hence acting on $X.$ Since $X$ admits
a Kähler-Einstein metric it also admits a\emph{ $T-$invariant }Kähler-Einstein
metric $\omega$ \cite{b-m}, which is the curvature form of a $T-$invariant
metric $\left\Vert \cdot\right\Vert $ on $-K_{X}.$ To simplify the
notation we assume that $-K_{X}$ is globally generated (otherwise
just replace $-K_{X}$ with a large tensor power) and take $p$ to
be the $T-$fix point and $s$ the holomorphic section of $-K_{X},$
furnished by Lemma \ref{lem:exist of section}. As explained above
\[
\phi:=-\log\left\Vert s\right\Vert ^{2}
\]
is then a strictly psh smooth exhaustion function of the Stein manifold
$S:=\{s\neq0\}$ containing the fixed point $p.$ Moreover, $\phi$
is $T-$invariant. Indeed, by construction $g\cdot s=\chi(g)s$ where
$\chi$ is a character on the compact group $T.$ Hence $\chi$ takes
values in the unit-circle $S^{1}\subset\C,$ showing that $\phi$
is $T-$invariant, as desired. 

As explained above the Kähler-Einstein equation for $\omega$ (and
the correspond metric on $-K_{X})$ then translates to the Monge-Ampère
equation on $S$ 
\begin{equation}
(dd^{c}\phi)^{n}=Ce^{-\phi}dV\label{eq:k-e equ in c^n}
\end{equation}
 for a positive constant $C,$ where $dV$ is the volume form induced
by the trivialization $s.$ We may rewrite $C=V_{X}/\int_{S}e^{-\phi}dV$,
where 
\[
V_{X}=\int_{S}(dd^{c}\phi)^{n}
\]
which coincides with the top-intersection number $c_{1}(X)^{n}.$
Let $\Omega_{R}$ be the set where $\phi<R$ and note that the sets
$\Omega_{R}$ exhaust $\C^{n},$ since $\phi$ is proper. Assume now,
to get a contradiction, that the bound in the theorem to be proved
is not valid for $X.$ Then we can fix $R$ sufficiently large so
that 
\[
V_{\Omega_{R}}:=\int_{\Omega_{R}}(dd^{c}\phi)^{n}>(n+1)^{n}.
\]
 Writing $\Omega:=\Omega_{R}$ and replacing $\phi$ by $\phi-R$
we then obtain a $T-$invariant smooth plurisubharmonic function $\phi$
solving the following equation: 
\begin{equation}
(dd^{c}\phi)^{n}=V_{\Omega}\frac{e^{-\phi}dV}{\int_{\Omega}e^{-\phi}dV},\,\,\,\,\phi=0\,\,\mbox{on\,\ensuremath{\partial\Omega}}\label{eq:m-a eq}
\end{equation}
on the $T-$invariant domain $\Omega$ which is a\emph{ hyperconvex
}domain, i.e. it admits a negative continuous psh exhaustion function
(namely $\phi).$ More precisely, since $\phi$ is smooth and strictly
psh the domain $\Omega$ is a\emph{ Stein domain}. Finally, rescaling,
i.e. replacing $\phi$ with $V_{\Omega}^{1/n}\phi$ we have thus obtained
a solution to the equation in Theorem \ref{thm:mean field intro}
with a parameter $\gamma:=V_{\Omega}^{1/n}>(n+1)$ and all that remains
is thus establishing Theorem \ref{thm:mean field intro} in the slightly
more general case when $\C^{n}$ is replaced with a Stein domain with
a holomorphic $S^{1}-$action admitting an attractive fixed point.

\subsubsection*{Proof of Theorem \ref{thm:mean field intro}}

We will denote by $\mathcal{H}_{0}(\Omega)$ the space of all psh
functions on $\Omega$ which are continuous up to the boundary, where
they are assumed to vanish. Its $T-$invariant subspace will be denoted
by $\mathcal{H}_{0}(\Omega)^{T}.$ Let $\phi_{0}$ be a solution to
equation \ref{eq:mean field type eq intro} with a fixed parameter
$\gamma.$ The following Moser-Trudinger type inequality can then
be established on $\Omega$ (which has $\phi_{0}$ as an extremal):
there is a positive constant $C$ such that 
\begin{equation}
\frac{1}{\gamma}\log\int_{\Omega}e^{-\gamma\phi}dV\leq\frac{1}{(n+1)}\int_{\Omega}(-\phi)(dd^{c}\phi)^{n}+C\label{eq:m-t ineq in proof}
\end{equation}
 for any $\phi\in\mathcal{H}_{0}(\Omega)^{T}.$ We also write the
previous inequality as $\mathcal{G}_{\gamma}(\phi)\leq C$ for the
corresponding Moser-Trudinger type functional $\mathcal{G}_{\gamma}.$
Given the convexity property in Theorem \ref{thm:prekopa type} the
proof may be obtained by repeating the proof of Theorem 1.4 in \cite{b-b},
but for completeness we have recalled the proof in section \ref{sub:Proof-of-the}
below. 

The desired contradiction will now be obtained by exhibiting a function
violating the previous Moser-Trudinger type inequality. To this end
we first recall the definition of the pluricomplex Green function
$g_{p}$ of a pseudoconvex domain $\Omega$ with a pole at a given
point $p:$ 
\begin{equation}
g_{p}:=\sup\left\{ \phi:\,\,\,\phi\in(PSH)(\Omega)\cap\mathcal{C}^{0})(\overline{\Omega}-\{p\}):\,\,\phi\leq0,\,\,\,\phi\leq\log|z|^{2}+O(1)\right\} \label{eq:g as sup}
\end{equation}
 where $z$ denotes fixed local holomorphic coordinates centered at
$p.$ We will often write $g=g_{p}$ to simpify the notation. As is
well-known \cite{kl} $g$ is continuous up to the boundary on $\Omega$
apart from a singularity at $p$ and satisfies 
\begin{equation}
(dd^{c}g)^{n}=\delta_{p}\,\mbox{\,\ on\,}\Omega-\{p\},\,\,\,\, g=\log|z|^{2}+O(1)\label{eq:eq for g}
\end{equation}
In particular $\int(dd^{c}g)^{n}=1$ and $\int_{\Omega}e^{-ng}dV=\infty.$
We note that if $T$ acts, as above, on $\Omega$ and $p$ is taken
as a fixed point for the action, then $g$ is $T-$invariant. Indeed,
since $p$ is invariant under the action of $T$ so is the the convex
class of functions where the sup in \ref{eq:g as sup} is taken and
hence the sup $g$ must be $T-$invariant. The contradiction is now
obtained by showing that there is a family of functions $g_{t}$ in
$\mathcal{H}_{0}(\Omega)^{T}$ decreasing to $g$ such that, for $t$
sufficently large $g_{t}$ violates the Moser-Trudinger type inequality
\ref{eq:m-t ineq in proof} if $\gamma>(n+1).$ To this end we set
\[
g_{t}:=\log(e^{-2t}+e^{g})-C_{t},\,\,\,\,\phi_{t}=(e^{-2t}+Ce^{\log|z|^{2}})
\]
where $C$ is a constant such that $-g_{t}\leq-\phi_{t}$ and $C_{t}$
is the constant ensuring that $g_{t}$ vanishes on the boundary, i.e.
$C_{t}=\log(e^{-2t}+1)=o(1/t).$ In fact, this part of the argument
does not require any symmetry assumptions. We fix local holomorphic
coordinates $z$ centered at $p$ such that, locally around $p,$
$\phi(z)=\log|z|^{2}+O(1).$ Trivially we have, by replacing $\Omega$
with a coordinate ball $B,$ that 
\[
-\frac{1}{\gamma}\log\int_{\Omega}e^{-\gamma g_{t}}dV\leq-\frac{1}{\gamma}\log\int_{B}e^{-\gamma\phi_{t}}dz\wedge d\bar{z}+C'
\]
Denoting by $F_{t}$ the scaling map defined by $F_{t}(z)=e^{t}z$
we have $\phi_{t}=F_{t}^{*}\phi_{0}-2t+o(1/t).$ Hence, we may write
\[
-\frac{1}{\gamma}\log\int_{B}e^{-\gamma\phi_{t}}dz\wedge d\bar{z}=-2t+o(1/t)+-\frac{1}{\gamma}\log\int_{B}e^{-\gamma F_{t}^{*}\phi_{0}}dz\wedge d\bar{z}
\]
Rewriting $dz\wedge\bar{dz}$ as $e^{-2tn}F_{t}^{*}dz\wedge\bar{dz}$
gives 
\[
-\frac{1}{\gamma}\log\int_{B}e^{-\gamma F_{t}\phi_{0}}dz\wedge d\bar{z}=2tn\frac{1}{\gamma}-I_{t},\,\,\, I_{t}:=\frac{1}{\gamma}\log\int_{e^{t}B}e^{-\gamma\phi_{0}}dz\wedge d\bar{z}
\]
 Next, we consider the energy term: since $-g_{t}\leq-\log(e^{-2t}+0)-C_{t}=2t+o(1/t)$
we get 
\[
\int_{\Omega}(-g_{t})(dd^{c}g_{t})^{n}\leq(2t+o(1/t))\int_{\Omega}(dd^{c}g_{t})^{n}.
\]
 But since $g_{t}$ is a sequence of bounded psh functions tending
to zero at the boundary of $\Omega$ and decreasing to $g$ it follows
from well-known convergence properties that $\int_{\Omega}(dd^{c}g_{t})^{n}\rightarrow\int_{\Omega}(dd^{c}g)^{n}=1$
(in fact, in our case $g$ may be taken to be smooth close to the
boundary and then the convergence follows immediately from Stokes
theorem). All in all this means that 
\[
-\mathcal{G}(g_{t})\leq\frac{1}{(n+1)}2t+(2t)(-1+\frac{n}{\gamma})-\log I_{t}+o(1/t),
\]
 Note that when $\gamma=(n+1)$ we have $(-1+\frac{n}{\gamma})=\frac{1}{(n+1)}$
and hence when $\gamma>(n+1)$ first term above, which is linear in
$t,$ tends to $-\infty$ as $t\rightarrow\infty.$ Moreover, since
$I_{t}\geq-C$ it follows that $-\mathcal{G}(g_{t})\rightarrow-\infty$
when $t\rightarrow\infty$ and hence $\mathcal{G}$ is not bounded
from above, which contradicts the Moser-Trudinger inequality. This
completes the proof of Theorem \ref{thm:mean field intro} and thus
of Theorem \ref{thm:max vol vector field intro}, as well.
\begin{rem}
There is an alternative way of obtaining a contradiction to the Moser-Trudinger
inequaly with parameter $\gamma>(n+1)$ (see our previous preprint
\cite{b-b1b}). Indeed, as shown in \cite{b-b} the inequality induces
another inequality of Brezis-Merle-Demailly type on the $(n+1)-$dimensional
product $\Omega'=\Omega\times D$ of $\Omega$ with the unit-disc
which in particular implies that 
\begin{equation}
\int_{\Omega'}e^{-(n+1)\phi}dV<\infty,\label{eq:dem inequality}
\end{equation}
 for any $S^{1}-$invariant plurisubharmonic function on $\Omega',$
say with isolated singularities compactly contained in $\Omega',$
vanishing on the boundary and with unit Monge-Ampère mass on $\Omega'.$
But this is immediately seen to be contradicted by the pluricomplex
Green function of $\Omega'$ with a pole at the origin. 
\end{rem}

\subsubsection{\label{sub:Symmetry-breaking}Symmetry breaking}

The assumption in Theorem \ref{thm:mean field intro} that $0$ be
contained in $\Omega$ is crucial. For example, when $\Omega$ is
an annulus, $r<|z|<1,$ in $\C,$ it is well-known \cite{clmp} that
there exists a (uniquely determined) $S^{1}-$invariant solution $\phi_{\gamma}$
for any value of $\gamma.$ Moreover, by the method of moving planes\emph{
any }solution of the equations is necesseraly $S^{1}-$invariant \cite{clmp}
and thus coincides with $\phi_{\gamma}.$ In the range $\gamma<2$
the solution $\phi_{\gamma}$ is an extremal for the corresponding
Moser-Trudinger inequalities which are known to hold for general funtions
$\phi$ in $\mathcal{H}_{0}(\Omega).$ We note however that at the
critical value $\gamma=2$ an interesting instance of symmetry breaking
appears. Indeed, since Theorem \ref{thm:prekopa type} applies for
any $\gamma>0,$ we deduce, as before, that $\phi_{\gamma}$ is always
an extremal for the corresponding Moser-Trudinger type inequality
for $S^{1}-$invariant functions in $\mathcal{H}_{0}(\Omega).$ But
when $\gamma>2$ the corresponding Moser-Trudinger inequality does
not hold on \emph{all} of $\mathcal{H}_{0}(\Omega).$ Indeed, as before
it is violated by a suitable regularization of a Green function with
a pole in any given point in $\Omega.$ This also shows that Theorem
\ref{thm:prekopa type} cannot hold general if the invariance assumption
is removed.

For completeness we will next recall the proof in \cite{b-b} of the
inequality used above. The arguments carry over verbatim to the setting
of Stein domains, even if as explained above the $\C^{n}$- setting
is adequate for our purposes.

\subsection{\label{sub:Proof-of-the}Proof of the Moser-Trudinger type inequality
\ref{eq:m-t ineq in proof}}

Let 
\[
\mathcal{G}(\phi):=\frac{1}{\gamma}\log\int_{\Omega}e^{-\gamma\phi}dV+\frac{1}{(n+1)}\int_{\Omega}\phi(dd^{c}\phi)^{n},
\]
 whose Euler-Lagrange equation (i.e. the critical point equation $d\mathcal{G}_{|\phi}=0)$
is precisely the complex Monge-Ampère equation \ref{eq:m-a eq}. Given
$\phi_{0}$ and $\phi_{1}$ in $\mathcal{H}_{0}(\Omega)$ there is
a unique \emph{geodesic} $\phi_{t}$ connecting them in $\mathcal{H}_{0}(\Omega).$
It may be defined as the following envelope: setting $\Phi(z.t):=\phi_{t}(z),$
where now $t$ has been extended to a complex strip $\mathcal{T}$
by imposing invariance in the imaginary $t-$direction, we set $M:=\Omega\times\mathcal{T}$
the boundary data $\Phi_{\partial M}$ is determined by $\phi_{0}$
and $\phi_{1}$ and we set 
\begin{equation}
\Phi(z,t):=\sup\left\{ \Psi(z,t):\,\,\,\Psi\in\mathcal{C}^{0}(\bar{M})\cap PSH(M),\,\,\,\Psi_{\partial M}\leq\Phi_{\partial M}\right\} \label{eq:geod as env}
\end{equation}
Since $M$ is hyperconvex it follows that $\Phi\in\mathcal{C}^{0}(\bar{M})\cap PSH(M)$
is the unique solution of the following Dirichlet problem:

\[
(dd^{c}\Phi)^{n+1}=0,\,\,\,\mbox{in\,\ \ensuremath{M}}
\]
and on $\partial M$ the function $\Phi$ coincides with the boundary
data determined by $\phi_{i}$ (see \cite{b-b} and references therein).
We also note that if $\phi_{t}$ is $T-$invariant for $t=0,1$ then
it is in fact $T-$invariant for any $t,$ as follows from its definition
\ref{eq:geod as env} as an envelope (just as in the similar case
of the Green function discussed above).

By Theorem \ref{thm:prekopa type} the functional 
\[
t\mapsto\frac{1}{\gamma}\log\int_{\Omega}e^{-\gamma\phi_{t}}dV
\]
 is concave along any geodesic (indeed, the assumption \ref{eq:prekopa type}
follows immediately from Taylor expanding a holomorphic function on
$\Omega$ around the origin and using the positivity of the weights
of the action). Next, we note that
\[
\mathcal{E}(\mathbf{\phi}):=\int_{\Omega}\phi(dd^{c}\phi)^{n}
\]
 is affine along a geodesic $\phi_{t}.$ Indeed, letting $t$ be complex
a direct calculation gives 
\begin{equation}
dd^{c}\mathcal{E}(\phi_{t})=\int_{\Omega}(dd^{c}\phi)^{n+1}\label{eq:psh-forward formula}
\end{equation}
which, by definition, vanishes if $\phi_{t}$ is a geodesic. All in
all this means that $\mathcal{G}(\phi_{t})$ is concave along a geodesic.
Letting now $\phi$ be an arbitrary element in $\mathcal{H}_{0}(\Omega)^{T}$we
take $\phi_{t}$ to be the geodesic connecting the solution $\phi_{0}$
of equation \ref{eq:m-a eq} (obtained from the invariant Kähler-Einstein
metric on $X)$ and $\phi_{1}=\phi.$ Heuristically, $\mathcal{G}(\phi_{t})$
has a critical point at $t=0,$ i.e. its right derivative vanishes
for $t=0$ and hence by concavity $\mathcal{G}(\phi_{1})\leq\mathcal{G}(\phi_{0}).$
However, since $\phi_{t}$ is not, a priori, smooth one has to be
a bit careful when differentiating $\mathcal{G}(\phi_{t}).$ However,
by the affine concavity of $\mathcal{E}$ it is not hard too see that
\[
\frac{1}{(n+1)}\frac{d}{dt}_{t=0^{+}}\mathcal{E}(\phi_{t})\leq\int_{\mathcal{B}}\dot{\phi}_{0}(dd^{c}\phi_{0})^{n}/n!
\]
(see Lemma 3.4 in \cite{b-b}) and hence, by concavity, $\mathcal{G}(\phi_{1})\leq0+\mathcal{G}(\phi_{0}),$
which concludes the proof of the M-T inequality with $C=\mathcal{G}(\phi_{0}).$ 
\begin{rem}
\label{rem:weakly Fano}Theorem is still valid when $X$ is merely\emph{
weakly Fano, }i.e $-K_{X}$ is nef and big, if one uses the notion
of (singular) Kähler-Einstein metrics introduced in \cite{bbegz}.
Indeed, by \cite{bbegz} such a metric is smooth on a Zariski open
subset of $X$ and hence, by the Kähler-Einstein equation, strictly
positively curved there. Moreover, by the Kawamata-Shokurov basepoint
free theorem (\cite{ko-mo}, Theorem 3.3), $-mK_{X}$ is base point
free for $m$ sufficently large and hence Lemma \ref{lem:exist of section}
still applies and together with the Bialynicki- Birula decomposition
one gets an exhaustion by $T-$invariant Stein domains of a dense
open embedding of $\C^{n}$ in $X.$ The rest of the proof then proceeds
exactly as before.
\end{rem}

\subsection{Bounds on the volume in the absense of a Kähler-Einstein metric.}

Recall that for a general Fano manifold $X$ the \textquotedblleft{}greatest
lower bound on the Ricci curvature\textquotedblright{} is the invariant
$R(X)\in]0,1]$ defined as the sup of all positive numbers $r$ such
that $\mbox{Ric \ensuremath{\omega\geq r\omega}}$ (see \cite{sz}).
A simple modification of the proof of Theorem \ref{thm:max vol vector field intro}
then gives the following more general statement:
\begin{thm}
\label{thm:R(X)}Let $X$ be a Fano manifold admitting a holomorphic
$\C^{*}-$action with a finite number of fix points. Then the first
Chern class $c_{1}(X)$ satisfies the following upper bound 
\[
c_{1}(X)^{n}\leq\left(\frac{n+1}{R(X)}\right)^{n}
\]

\end{thm}
The point is that, as shown in \cite{sz}, the invariant $R(X)$ coincides
with the sup of all $r\in[0,1[$ such that, for any given Kähler form
$\eta$ there exists $\omega_{r}$ such that 
\[
\mbox{Ric \ensuremath{\omega_{r}=r\omega_{r}+(1-r)\eta}}
\]
As is well-known $\omega_{r}$ is uniquely determined. In particular,
if $X$ admits a holomorphic action by $S^{1}$ then, by averaging
over $S^{1},$ we may take $\eta$ to be $S^{1}-$invariant. By the
uniqueness of solutions to the previous equation it then follows that
$\omega_{r}$ is also $S^{1}-$invariant. For any fixed $r$ we can
now repeat the proof of Theorem \ref{thm:max vol vector field intro}
and obtain an $S^{1}-$invariant psh solution $\phi$ to 
\begin{equation}
(dd^{c}\phi)^{n}=V_{\Omega}\frac{e^{-r\phi}e^{-(1-r)\phi_{0}}dV}{\int_{\Omega}e^{-r\phi}e^{-(1-r)\phi_{0}}dV},\,\,\,\,\phi=0\,\,\mbox{on\,\ensuremath{\partial\Omega}},\label{eq:m-a eq-with t}
\end{equation}
 where $\eta=dd^{c}\phi_{0}$ on $\Omega.$ The Moser-Trudinger inequality
still applies when $dV$ is replaced with $e^{-(1-r)\phi_{0}}dV.$
Indeed, we just apply Theorem \ref{thm:prekopa type} to the curve
$(1-r)\phi_{t}+r\phi_{0},$ where $\phi_{t}$ is a geodesic as before.
In fact, since $\phi_{0}$ is bounded this gives the same Moser-Trudinger
inequality as before, but with the integrand $e^{-\phi}$ replaced
with $e^{-r\phi}.$ Hence, rescaling we obtain a contradiction if
$r^{n}V(X)>(n+1)^{n},$ just as before. In particular, taking the
sup over all $r<R(X)$ then concludes the proof of the previous theorem.

\section{\label{sec:The-volume-of-singular-toric}The volume of (singular)
Toric varieties and convex bodies}

We will next briefly explain how to carry out the proof of Theorem
\ref{thm:max vol intro toric sing} possibly \emph{singular} toric
Fano varieties directly using convex analysis in $\R^{n}.$ This approach
has the advantage of bypassing some technical difficulties related
to the singularities of the toric variety in question. The motivation
comes from the well-known correspondence between $T-$invariant positively
curved metrics on toric line bundles $L\rightarrow X_{P}$ and convex
functions in $\R^{n}$ whose (sub-gradient) image is contained in
the corresponding polytope $P$ (see \cite{do,b-b-2} and references
therein) More precisely, a psh function $\Phi(z)$ on the complex
torus $\C^{*n},$ embedded in $X,$ is the weight of a, possibly singular,
positively curved metric on $L\rightarrow X_{P}$ iff $\Phi$ is the
pull-back under the Log map 
\[
\mbox{Log }\C^{*n}\rightarrow\R^{n}:\,\,\,\, z\mapsto x:=(\log(|z_{1}|^{2},...,\log(|z_{n}|^{2}),
\]
of a convex function $\phi(x)$ on $\R^{n}$ such that the (sub-)
gradient image $d\phi(\R^{n})$ is contained in $P.$ This correspondence
has been mostly studied in the case when $X_{P}$ is smooth and the
metric on $L$ is smooth and positively curved. Then the corresponding
convex function has the property that the differential (gradient)
$d\phi$ defines a diffeomorphism from $\R^{n}$ to the interior of
$P$ and moreover its Legendre transform satisfies Guillemin's boundary
conditions (see \cite{do} and references therein). However, we stress
that these refined regularity properties will play no role here. Our
normalizations are such that 
\[
\mbox{(Log)}_{*}MA(\Phi)=MA_{\R}(\phi),
\]
 where $MA_{\R}(\phi)$ denotes $n!$ times the usual real Monge-Ampère
measure of the convex function $\phi,$ i.e. $MA_{\R}\phi)(E):=n!\int_{d\phi(E)}dp,$
for any Borel measure $E.$ 

We also point out that in the case when $X_{P}$ is\emph{ smooth }one
may assume that $P$ is contained in the positive octant with a vertex
at $0$ and the psh function $\Phi$ can then by taken to be defined
on all of $\C^{n},$ as in the proof of Theorem \ref{thm:max vol vector field intro}.
The technical difficulity in the general case when $X_{P}$ may be
singular - when seen from the complex point of view - is that, even
if we may still arrange that $P$ is in the positive octant with a
vertex at $0,$ the function $\phi$ will a priori only be continuous
on $\C^{*n},$ i.e. away from the coordinate axes. 

Anyway, from the\emph{ real} point of view the latter theorem can
be rephrased entirely in terms of convex analysis on convex domains
in $\R^{n}$ of the form 
\[
\Omega:=\{\psi<0\}
\]
 where $\psi$ is a convex function on $\R^{n}$ such that its (sub-)
gradient image of a convex body $P$ in the positive octant. We propose
to call such a domain \emph{monotone. }Note that, since $\partial\psi/\partial x_{j}\geq0$
any monotone convex domain $\Omega$ is invariant under the action
of the additive semi-group $]-\infty,0]^{n}$ by translations. In
particular, $\Omega$ is \emph{unbounded} in contrast to the standard
setting of bounded convex domains in convex analysis. Still, it is
not hard to generalize the usual properties valid for convex functions
on bounded domains, as long as one works with \emph{bounded} convex
functions on $\Omega.$ More precisely, a convenient function space
to work with is the space $\mathcal{H}(\Omega,\psi)$ of all bounded
convex functions $\phi$ on $\Omega$ such that $\phi\geq\delta\psi$
for some positive number $\delta$ (depending on $\phi).$ For example,
it is not hard to see that for any such $\phi$ we have $\phi=0$
on $\partial\Omega$ and the total real Monge-Ampère mass of $\phi$
on $\Omega$ is finite. Now all the usual notions concerning psh functions
on bounded pseudoconvex domains can be transported to the setting
of monotone convex domains. For example, we can define the \emph{Green
function }$g$ of $\Omega$ (playing the role of the usual pluricomplex
Green function for a domain in $\C^{n}$ with a pole at the origin)
by 
\begin{equation}
g(x):=\sup\left\{ \phi(x):\,\,\,\phi\in\mathcal{H}(\Omega,\psi):\,\,\phi(x)\leq\log(\sum_{i=1}^{n}e^{x_{i}})+O(1)\right\} \label{eq:def of convex green f}
\end{equation}
Then standard arguments show that 
\begin{itemize}
\item $g_{\Omega}$ is convex and continuous on $\bar{\Omega}$ and $g_{\Omega}=0$
on $\partial\Omega$ 
\item $g(x)=\log(\sum_{i=1}^{n}e^{x_{i}})+O(1)$
\item $MA_{\R}(g)=0$ and if $g_{t}:=\log(e^{-2t}+e^{g_{\Omega}})$ then
$g_{t}\in\mathcal{H}(\Omega,\psi)$ with $\lim_{t\rightarrow\infty}\int_{\Omega}MA_{\R}(g_{t})=1$
and 
\begin{equation}
\lim_{t\rightarrow\infty}\int_{\Omega}e^{-ng_{t}}d\nu_{n}=\infty,\label{eq:non-integr of green}
\end{equation}

\end{itemize}
where 
\begin{equation}
d\nu_{n}(x)=e^{\sum_{i=1}^{n}x_{i}}dx,\label{eq:push-forw of lesb}
\end{equation}
 i.e. $d\nu_{n}$ is a multiple of the push-forward under the Log
map of the Lebesgue measure on $\C^{n}.$ 

Moreover, if $\Phi$ is a $T^{n}-$invariant psh function in $\C^{n}$
satisfying the\emph{ Kähler-Einstien equation} \ref{eq:k-e equ in c^n}
then its push-forward $\psi$ satisfies the following real Monge-Ampère
equation 
\begin{equation}
MA_{\R}(\psi)=e^{-\psi(x)}d\nu_{n}(x),\label{eq:real k-e equ with lesb}
\end{equation}
The following theorem, proved in \cite{b-b-2}, shows that the previous
equation admits a solution iff $(1,...1)$ is the barycenter of $P:$
\begin{thm}
\label{thm:-existence of real k-e}\cite{b-b-2} Let $P$ be a convex
body and fix an element $p_{0}$ in $P.$ Then there is a smooth convex
function $\phi$ such that $d\phi$ defines a diffeomorphism from
$\R^{n}$ to the interior of $P$ and such that $\phi$ solves the
equation 
\[
MA_{\R}(\phi)=e^{-\phi(x)+\left\langle p_{0},x\right\rangle }dx
\]
 on $\R^{n}$ iff $p_{0}$ is the barycenter of $P.$ 
\end{thm}
One can now go on to obtain Moser-Trudinger type inequalities (and
Brezis-Merle-Demailly type inequalities) essentially as before, using
Prekopa's convexity theorem in $\R^{n}.$ The only technical difficulty
is to make sure that the corresponding energy type functional 
\[
\mathcal{E}_{\R}(\phi):=\frac{1}{(n+1)}\int_{\Omega}\phi MA_{\R}(\phi)
\]
still has the appropriate properties (for its first and second derivatives
along geodesics). To this end one can first establish that if $M(\phi_{1},...,\phi_{n})$
denotes the real mixed Monge-Ampère measure obtained by polarizing
the usual real Monge-Ampère measure, then the pairing 
\[
(\phi_{0},\phi_{1},...,\phi_{n})\mapsto\int_{\Omega}\phi_{0}M(\phi_{1},...,\phi_{n}),
\]
is finite and symmetric on $\mathcal{H}(\Omega,\psi)^{n+1}.$ The
new technical difficulty compared to the classical situation (see
for example \cite{kla} and references therein) comes from the unboundedness
of $\Omega.$ But using that $\Omega$ is \emph{complete} ``at infinity''
in $\Omega$ (wrt the Euclidean metric) one can use standard cut-off
function argument to carry out the required integrations by parts.
Concretely, given a smooth compactly supported function $f$ on $\R$
one can use the sequence $\chi_{j}(x):=f(|x|/j)^{2}$ as cut-off functions
on $\Omega$ (this is sometimes referred to as ``the Gaffney trick'').

The advantage of the $\R^{n}-$approach is that it applies equally
well to the case when the corresponding toric variety $X_{P}$ is
singular. In fact, one may as well replace $P$ with any convex body,
even though there is then no corresponding toric variety. Mimicking
the proof of Theorem \ref{thm:max vol vector field intro}, replacing
$\phi$ with a solution of the equation in Theorem \ref{thm:-existence of real k-e}
one then obtains a proof of Theorem \ref{thm:max vol for convex body intro}
essentially as above.

Next, we explain how to deduce Cor \ref{cor:ehrhart intro} from Theorem
\ref{thm:max vol for convex body intro} Given a real polytope $P$
with non-empty interior we can write it as 
\[
P=\{p\in\R^{n}:\,\,\alpha_{F}(p)\geq0\},
\]
where $F$ is an index running over the facets $\{\alpha_{F}=0\}$
of $P.$ We fix a vertex $v$ and affine functions $\alpha_{F_{1}},....,\alpha_{F_{N}}$
cutting out $n$ faces of $P$ meeting $v$ and spanning a cone of
maximal dimension. Then 
\[
P'=\alpha(P),\,\,\,\,\,\,\alpha(p):=(\alpha_{F_{1}}(p)/\alpha_{F_{1}}(b_{P}),...,\alpha_{F_{n}}(p)/\alpha_{F_{n}}(b_{P}))
\]
is an $n-$dimensional polytope in the positive octant $[0,\infty[^{n}$
such that $b_{P'}=(1,1,...,1).$ Moreover, if $P$ is a rational polytope
as in the statement of Cor \ref{cor:ehrhart intro}, then 
\[
\mbox{Vol}(P')=\frac{d}{a_{F_{1}}\cdots a_{F_{n}}}\mbox{Vol}(P)
\]
where $d$ is the determinant of the linear map $p\mapsto(l_{F_{1}}(p),...,l_{F_{n}}(p)).$
But the map is represented by an invertible matrix with integer coefficients
and hence $d$ is a positive integer and in particular $d\geq1.$
Since, by assumption, $a_{F_{i}}\leq1$ it follows that $\mbox{Vol}(P')\geq\mbox{Vol}(P)$
and hence it will be enough to prove Theorem \ref{thm:max vol for convex body intro}.
\begin{rem}
\label{rem:grun}As kindly pointed out to us by Bo'az Klartag Theorem
\ref{thm:max vol for convex body intro} can also be deduced from
\emph{Grunbaum's inequality.} We thank him for allowing us to reproduce
his elegant argument here. First recall that the Grunbaum inequality
says that if $P$ is a convex body, and $P_{-}$ denotes the intersection
of $P$ with an affine half-space defined by one side of a hyperplane
$H$ passing through the barycenter of $P,$ then 
\[
\mbox{Vol}(P_{-})\geq\left(\frac{n}{(n+1)}\right)^{n}\mbox{Vol}(P).
\]
 In particular, if $P$ is a convex body as in the statement of Theorem
\ref{thm:max vol for convex body intro} we can take $P_{-}$ to be
$n$ times the unit-simplex $\Delta$ in the positive octant with
one vertex at $0.$ Since, $\mbox{Vol}(P_{-})\leq\mbox{Vol}(n\Delta)=n^{n}/n!$
this gives the desired inequality.
\end{rem}

\section{\label{sec:Symplectic-geometry,-multiplicity}Symplectic geometry
and multiplicity free actions }

\subsection{\label{sub:The-homogeneous-case}The homogeneous case}

Let us start by specializing Theorem \ref{thm:max vol vector field intro}
to the case when $X$ is a rational homogenuous space. Even though
this is the simples case the corresponding degree bound, when rephrased
in terms of representation theory, is highly non-trival and was first
obtained by Snow \cite{sn}. 

Let us first recall some basic representation theory (see \cite{sn}
and references therein). Let $K$ be a compact complex semi-simple
Lie group and denote by $G$ its complexification. Under the adjoint
action of a fixed maximal torus $T$ the Lie algebra $L(G)$ decomposes
as 
\[
L(G)=L(T_{c})\oplus E_{+}\oplus\overline{E}_{+},\,\,\, E_{+}=\bigoplus_{\alpha\in R^{+}}E_{\alpha},
\]
 where $R_{+}$ is a consistent choice of \emph{positive roots }$\alpha\in L(T)^{*}$
and $E_{\alpha}$ denote the corresponding weight spaces, i.e. $E_{\alpha}$
is generated by a vector $Z_{\alpha}$ such that $[t,Z_{\alpha}]=i\left\langle t,\alpha\right\rangle Z_{\alpha}$
for any $t\in L(T).$ A consistent choice of positive roots $R_{+}$
corresponds to a choice of\emph{ Borel group} $B$ with Lie algebra
\[
L(B)=L(T)\oplus\overline{E}_{+}
\]
The corresponding (complete) \emph{flag variety} is defined as the
$G-$homogenuous compact complex manifold $G/B(=K/T).$ As is well-known
any rational $G-$homogeneous compact complex manifold $X$ may be
written as 
\[
X=G/\mathcal{P},
\]
 for some $G$ as above and a \emph{parabolic} subgroup $\mathcal{P}$
of $G$ (i.e. a subgroup containing a Borel group which we may assume
is $B).$ We recall that any $G-$homogeneous line bundle $L\rightarrow G/\mathcal{P}$
is determined by a \emph{weight} $\lambda,$ i.e. an element in the
weight lattice of $L(T_{c})^{*}.$ Indeed, 
\[
L_{\lambda}=G\times_{(P,\rho_{\lambda})}\C,
\]
where $\rho_{\lambda}$ is a homomorphism $P\rightarrow\C^{*},$ which
is uniquely detetermined by its restriction to the complex torus in
$P$ and hence determined by an element $\lambda$ in the weight lattice
of $L(T_{c})^{*}.$ By construction, the tangent bundle $TX$ is generated
at the identity coset by root vectors $Z_{\alpha}\in E_{\alpha}$
for $\alpha$ in a subset $R_{X}^{+}$ of the positive roots. In particular,
the weight $\lambda_{X}$ of the anti-canonical line bundle $-K_{X}$
may be written as 
\begin{equation}
\lambda_{X}=\sum_{\alpha\in R_{X}^{+}}\alpha,\label{eq:weight of canoni}
\end{equation}
defining an ample line bundle, so that $X$ is Fano. Moreover, by
the general Weyl character formula, 
\[
c_{1}(L_{\lambda})^{n}=n!\prod_{\alpha\in R_{X}^{+}}\frac{\left\langle \alpha,\lambda\right\rangle }{\left\langle \alpha,\rho\right\rangle },
\]
with $\rho$ denoting, as usual, the half-sum of all the positive
roots $\alpha$ and where $\left\langle \cdot,\cdot\right\rangle $
denotes the Killing form. Hence, Theorem \ref{thm:max vol vector field intro}
applied to $G/\mathcal{P}$ translates to a Lie algebra statement
first shown by Snow \cite{sn}. Indeed, as shown in \cite{sn} (Prop
1) formula \ref{eq:weight of canoni} can be simplified using the
Dynkin diagram of $G.$ Using this latter formula and the Weyl character
formula above together with classification theory for semi-simple
Lie groups, Snow then shows, using quite elaborate calculations, how
to deduce the bound in Theorem \ref{thm:max vol vector field intro}
for the Fano manifold $X=G/\mathcal{P}.$ It should also be pointed
out that the results in \cite{sn} also give that $\P^{n}$ is the
\emph{unique} maximizer of the degree among all partial flag manifolds.

\subsection{The case of spherical varities and multiplicity free actions}

In this section we will reformulate Theorem \ref{thm:max vol vector field intro}
in the case when $X$ is spherical variety, using Brion's formula
for the volume \cite{bri-1} of a line bundle on $X.$ We will use
the symplecto-geometric formulation (see the end of \cite{bri-1}
and \cite{br-0} where further references can be found).

Let us start by recalling the definition of the \emph{Duistermaat-Heckman
measure} in symplectic geometry. Let $(X,\omega)$ be a symplectic
manifold and $K$ a compact connected Lie group acting by symplectomorphisms
on $X$ and fix a compact maximal torus $T$ in $K.$ Assume for simplicity
that the first Betti number of $X$ vanishes (which will be the case
here since $X$ will be a Fano manifold). Then there is a\emph{ moment
map} 
\[
\mu_{T}:\,\, X\rightarrow L(T)^{*},
\]
 where $L(T)$ denotes the real vector space given by the Lie algebra
of $T.$ The image $P:=\mu_{T}(X)$ is a convex rational polytope
and the density $v(p)$ of the\emph{ Duistermaat-Heckman measure }$(\mu_{T})_{*}\omega^{n}/n!$
on $\mu_{T}(X)$ is continuous and piecewise polynomial. The convex
polytope obtained by intersecting $\mu_{T}(X)$ with a fixed positive
Weyl chamber $\Lambda_{+}$ in $L(T)^{*}$ is called the\emph{ moment
polytope.}

The action of $K$ is said to be \emph{multiplicity-free }if the group
of symplectomorphisms of $(X,\omega)$ which commute with $K$ is
abelian ( equivalently, all $K-$invariant functions on $X$ Poisson
commute). As is well-known \cite{br-0} any $G-$spherical non-singular
complex algebraic variety $X$ with a $G-$equivariant ample line
bundle $L\rightarrow X$ equipped with a fixed positive curvature
form $\omega$ in $c_{1}(L)$ corresponds to a symplectic manifold
$(X,\omega)$ with a symplectic action by the real form $K$ of $G$
which is multiplicity-free\emph{.} Moreover, as shown by Brion \cite{bri-1},
in the spherical (i.e. multiplicity free) case the density $v$ of
the Duistermaat-Heckman measure is explicitly given by 
\[
v(p)=\prod_{\alpha}\frac{\left\langle \alpha,p\right\rangle }{\left\langle \alpha,\rho\right\rangle },
\]
where $\rho$ denotes the half-sum of all the positive roots $\alpha$
and the products runs over all positive roots $\alpha$ such that
$\left\langle \alpha,p\right\rangle >0.$ Moreover, the Lesbesgue
measure $dp$ has been normalized to give unit-volume to the fundamental
domain of the lattice in $P$ (see section \ref{sub:The-homogeneous-case}
). By the definition of $v$ as the density of the Duistermaat-Heckman
measure 
\[
\int_{X}\omega^{n}/n!=\int_{P}v(p)dp
\]
Hence, applying Theorem \ref{thm:max vol vector field intro} to a
spherical non-singular Fano variety gives the following
\begin{cor}
Let $(X,\omega)$ be symplectic manifold with a symplectic action
by a compact Lie group $K$ which is multiplicity-free and denote
by $P$ the image of the moment map associated to a fixed maximal
torus $T$ in $K.$ If $X$ admits an integrable $\omega-$compatible
complex structure $J$ preserved by $K$ and such that the cohomology
class $[\omega]$ contains a Kähler-Einstein metric on $(X,J)$ then
\[
\int_{P}v(p)dp\leq(n+1)^{n}/n!
\]
 Equality hold when $(X,\omega)$ is complex projective space equipped
and $\omega$ is the standard suitably normalized $SU(n)-$invariant
symplectic form (i.e. $\omega$ is $(n+1)$ times the Fubini-Study
form).
\end{cor}
We recall that the condition that a spherical variety $X$ be Fano
can be expressed rather explicitly in algebro-geometric terms \cite{bri-2}.
\begin{rem}
\label{rem:finite}It is well-known that any spherical variety $X$
has a holomorphic $\C^{*}-$action such that the fixed point set.
$X^{\C^{*}}$ is finite, so that Theorem \ref{thm:max vol vector field intro}
indeed can be applied to $X.$ Let us briefly recall the reason that
$X^{\C^{*}}$ is finite (as kindly explained to us by Michel Brion).
The starting point is the fundamental fact that any $G-$spherical
variety $X$ can be covered by a finite number of $G-$orbits \cite{l-v}.
Next one shows that if $G$ is a reductive group acting transitively
on a set $Y$ (here an orbit of $G)$ then the fixed point set $Y^{T_{c}}$
is finite for any maximal complex torus $T_{c}$ in $G.$ Indeed,
the Weyl group $N_{G}(T_{c})/T_{c}$ is finite if $G$ is reductive
and its acts transitively on $Y^{T_{c}},$ as follows from the definition
of the normalizer $N_{G}(T_{c})$ (see Proposition 7.2 in \cite{dec-pr}).
Finally, it is a general fact that for a generic (regular) one-parameter
subgroup $\C^{*}$ in $T_{c}$ one has that $Y^{T}=Y^{\C^{*}}$ (as
can be proved by reducing the problem to the linear action of $T$
on a vector space using \cite{su})
\end{rem}
In particular, the previous corollary applies to\emph{ horospherical
}Fano varieties. These are homogenous toric bundles over a rational
homogenous variety \cite{po-s,pas}. Let us for simplicity consider
the case of homogenous fibrations $X$ over a (complete) flag variety:
\[
X\rightarrow G/B
\]
 where any fiber is biholomorphic to a given toric variety $F.$ Then
the corresponding polytope $P$ is contained in the interiour of a
positive Weyl chamber, coinciding with the moment polytope of $F$
under the induced torus action. Moreover, $X$ is Fano with a Kähler-Einstein
metric iff $F$ is Fano and the sum of the positive roots $\sum\alpha$
coincides with the barycenter of the reflexive polytope $P$ wrt the
Duistermaat-Heckman measure $vdp$ (this was first shown in \cite{po-s},
but see also the illuminating discussion in section 4.1 in \cite{do}).
Hence, we arrive at the following
\begin{cor}
\label{cor:horo}Let $G$ be a semi-simple complex Lie group and fix
a maximal torus $T$ in $G$ and a set $R^{+}$ of $n$ positive roots
$\alpha_{i}$ for the Lie algebra of $G.$ Let $P$ be a reflexive
lattice polytope in the positive Weyl chamber of $L(T)^{*}$which
is Delzant and such that $\sum_{\alpha\in R^{+}}\alpha$ is the barycenter
of $P$ wrt the Duistermaat-Heckman measure $vdp.$ Then 
\[
\int_{P}vdp\leq(n+1)^{n}/n!
\]

\end{cor}
We expect that the condition that $P$ be Delzant, i.e. the corresponding
toric variety is smooth, can be removed.

Finally, let us mention the connection to\emph{ Okounkov bodies. }As
shown by Okounkov \cite{ok1} one can associate another convex polytope
$\Delta$ to a polarized spherical variety $X,$ such that $\Delta$
fibers over the moment polytope $P$ (the fibers being the Gelfand-Cetlin
string polytopes). The definition is made so that 
\[
c_{1}(L)^{n}/n!=\mbox{Vol}(\Delta)
\]
 More generally, to any polarized projective variety $(X,L)$ there
is convex body $\Delta$ associated (further depending on an auxiliary
choice of flag in $X)$ such that the previous formula for $c_{1}(L)^{n}/n!$
holds \cite{l-m,ka}. In the light of Theorem \ref{thm:max vol vector field intro}
and the toric case it would be interesting to know if the condition
that $X$ be Fano (and $L=-K_{X})$ with a Kähler-Einstein metric
can be naturally expressed in terms of properties of $\Delta?$ See
also \cite{a-k} for the case of reductive spherical varieties.

\end{document}